\documentclass[a4paper,12pt]{article}
\usepackage{fullpage}
\usepackage{amsfonts}
\usepackage{amssymb,amsthm,amsmath,color}
\usepackage{array}
\usepackage{mathrsfs}
\usepackage{dsfont}
\usepackage{lmodern}
\usepackage{bbm}
\usepackage{hyperref}
\usepackage{graphicx}
\usepackage[utf8]{inputenc}  
\usepackage[T1]{fontenc} 
\usepackage{stmaryrd}  
\usepackage{tikz}
\usepackage{tkz-euclide}
\usepackage{biblatex}
\numberwithin{equation}{section}

\DeclareMathAlphabet{\mathbbo}{U}{bbold}{m}{n}

\newtheorem{theorem}{Theorem}[section]

\newtheorem{lem}[theorem]{Lemma}

\newtheorem{cor}[theorem]{Corollary}
\newtheorem*{question}{Question}

\newenvironment{rmq}{\stepcounter{theorem}\noindent\textbf{Remark \thetheorem}.}{}

\newcommand{\N}{\mathbb N}

\newcommand{\R}{\mathbb R}
\newcommand{\D}{\mathbb D}

\title{Existence of a Sidon set for the distinct distance constant}

\author{Robin Riblet\thanks{\textsc{ENS Paris-Saclay, Centre Borelli, UMR 9010, 91190 Gif-sur-Yvette, France}, robin.riblet@ens-paris-saclay.fr} , Titien Schehr\thanks{\textsc{ENS Paris-Saclay, Centre Borelli, UMR 9010, 91190 Gif-sur-Yvette, France}, titien.schehr@ens-paris-saclay.fr}}
\begin{document}

    \maketitle

  \begin{abstract}
We highlight a certain compactness of Sidon sets and $B_2[g]$-sets and provide several applications. Notably, we prove the existence of such sets that maximize certain functions. In particular, we show the existence of a Sidon set whose reciprocal sum is equal to the distinct distance constant. We also improve the best known bounds for this constant.

\end{abstract}

\section*{Introduction}

A Sidon set, or $B_2$-sequence, is a sequence of positive integers $(s_i)$ such that all the sums $s_i+s_j$ $(i \leq j)$ are distinct. These are well-studied sets, and it's well known, following the work of Singer, Erd\H{o}s, Bose, and Chowla (see \cite{OBryant}), that the maximum cardinality of a Sidon set in the first $n$ integers is of order $\sqrt{n}$. It is legitimate to ask whether there is an infinite Sidon set of density of order $n^{-1/2}$. Erd\H{o}s has shown that the answer is no (see \cite{HalberstamRoth} Theorem 8 p.89), and the best density currently known is due to Ruzsa \cite{Ruzsa}, who proves the existence of a Sidon set of density of order $n^{\sqrt{2}-2}$. This result was a real breakthrough, since until then, and for more than 50 years, the best known density had been given by the greedy algorithm $G$, known as the Mian–Chowla sequence. Its density is only of order $n^{-2/3}$, which is less than the density established by Ruzsa, but this sequence has other advantages. For a long time, this was the $B_2$-sequence with the largest known reciprocal sum $S_G=\sum_{g\in G}1/g$ (see \cite{Guy} p.351). Lewis (see \cite{Lewis} p.164-165) found that
$$
2.158435 \leq S_G \leq 2.158677.
$$
In 1991, Zhang \cite{Zhang} found a Sidon sequence with a reciprocal sum greater than $S_G$. Zhang’s sequence $Z$ is obtained by applying the greedy algorithm while imposing that the 15th term be 229. Zhang proved that
$$
S_Z>2.1597.
$$
In 2015, by repeating Zhang's construction but setting $Z_{27}=962$, Salvia \cite{Salvia} proves the existence of a new Sidon set $H$ such that
$$S_Z < 2.16027651 < S_H<2.16028417  .$$
Very recently (June 2025), L. Kleinwaks\footnote{see \url{https://oeis.org/A384729}} exhibited a Sidon set $K$ of 1010 terms with a reciprocal sum greater than $2.16150003$. As he himself points out, this set can be supplemented by the greedy algorithm, which improves this bound to $2.1615001$. This is the best known bound currently. The distinct distance constant (DDC) is the supremum of the set of the reciprocal sums of Sidon sequences. 
$$\mathrm{DDC}=\sup\left\{ \sum_{s\in S}1/s \ : \ S\subset\N \text{ is a Sidon set} \right\}.$$
Until now, we did not know whether this bound could be reached, i.e. whether there exists a Sidon set with a reciprocal sum equal to the DDC (see \cite{Salvia} p.1). We show in this paper that it does.

 \setcounter{section}{1}
\setcounter{theorem}{3}
\begin{theorem}
    There exists a Sidon set $S\subset\N$ such that 
    $$\sum_{s\in S}\frac{1}{s}=\mathrm{DDC}.$$
\end{theorem}

In fact, we will prove a more general result that directly implies it.

    \setcounter{section}{1}
\setcounter{theorem}{1}

\begin{theorem}
    Let $\alpha> \frac{1}{2}$, there exists a Sidon set $S_\alpha\subset\N$ such that 
    $$\sum_{s\in S_\alpha}\frac{1}{s^\alpha}=\sup\left\{ \sum_{s\in S}s^{-\alpha} \ : \ S\subset\N \text{ is a Sidon set}\right\}.$$
\end{theorem}
In Section \ref{sectionProprietes}, we establish some properties of these maximal sets, but we do not know if there is a unique Sidon set achieving the maximum. $H$ and $Z$ share the first 27 terms in common, but until now, we only knew from Taylor and Yovanof \cite{Taylor} that if such a set existed (what this article proves), then it must begin with the values $1,2,4$. 

In \cite{Levine}, Levine showed that $\mathrm{DDC}<2.37366$.
The upper bound was improved by Taylor \cite{Taylor}, who obtained $\mathrm{DDC}<2.24732646$. In Section \ref{sectionMajDDC}, we enhance this bound by adopting his approach and by using the well-known upper bound of Lindström \cite{Lindstrom} on the cardinality of a Sidon set within the first $n$ integers.

    \setcounter{section}{5}
\setcounter{theorem}{0}
\begin{theorem}\label{majDDC}
    We have $2.1615001<\mathrm{DDC}\leq 2.247307$.
\end{theorem}

Our method, developed in Section \ref{sectionMainResult}, can be widely adapted to other types of functions to be maximized and to other patterns than Sidon sets. In the last Section (Section 6), we give an equivalent to Theorem \ref{MainResult} for sum-free sets (ie. sets with no solution to the equation $x+y=z$).

 \setcounter{section}{6}
\setcounter{theorem}{1}
\begin{theorem}
    Let $\alpha\in\R$. There exists $F_\alpha\in\mathcal F$ such that
    $$\sum_{f\in F_\alpha} f^{-\alpha}=\sup\left\lbrace \sum_{f\in F} f^{-\alpha} \ : \ F\in\mathcal F\right\rbrace,$$
    where $\mathcal{F}$ is the set of the sum-free sets in $\N^*$.
\end{theorem}

Then we discuss a weaker adaptation of our method to obtain the existence of sets maximizing any continuous function on any closed subset of $\mathcal{P}(\N)$ with the discrete product topology ($\mathcal{P}(\N)$ seen as $\left\lbrace 0,1\right\rbrace^\N$).
 
\begin{theorem}
    Let $\mathcal{A}$ be a closed part of $\mathcal P(\N)$ and $\mathcal{F}:\mathcal A \rightarrow \R$ continuous, then there exists $A\in\mathcal A$ such that $\mathcal F(A)=\sup \mathcal{F}(\mathcal A)$.
\end{theorem}
In particular, this allows us to extend Theorem \ref{MainResult} to all $B_h[g]$.

\begin{cor}
Let $g\geq 2$, $h \geq 2$, and $\alpha> \frac{1}{h}$, there exists $B_\alpha\in B_h[g]$ such that
    $$\sum_{b\in B_\alpha}\frac{1}{b^\alpha}=\sup\left\{ \sum_{b\in B} b^{-\alpha} \ : \ B\in B_h[g]\right\}<+\infty.$$
\end{cor}

   \setcounter{section}{0}

\section{Main results}\label{sectionMainResult}

Let $\mathbb{D}$ denote the open unit disk in the complex plane, and let $\mathcal{O}(\D)$ be the complex vector space of analytic functions on this disk, with the topology of uniform convergence on compact sets. Let $\mathcal{S}\subset\mathcal{P}(\N^*)$ be the set of all Sidon sets and for all $g\geq 1$ and $h\geq 2$, $B_h[g]$ be the set of subsets of integers that, for all \( k \in \mathbb{N} \), admit at most \( g \) solutions $\{x_1,...,x_h\}$ to the equation \( x_1 + ...+x_h = k \). We obviously have $\mathcal{S}=B_2[1]$, but we will retain the notation \( S \) for readers who are not familiar with the \( B_2 \)-sets. For all $B\subset\mathbb{N}$, we denote
$$f_B(z)=\sum_{n=0}^\infty{\mathbbm{1}_B(n)z^n} ,$$
the generating function of $B$. We are going to show that
the set of generating functions of Sidon sets is compact.

\begin{theorem}\label{mainResult}
The set $\mathfrak S=\left\lbrace z\mapsto f_A(z) \ : \ A\in \mathcal{S}\right\rbrace$ is a compact subset of $\mathcal{O}(\D)$ with the
topology of uniform convergence on compact subsets. Moreover, for all $f\in \mathcal{C}^{0}([0,1[)$ such that $\int_{0}^1{\frac{|f(t)|\sqrt t}{\sqrt{1-t}}}dt < +\infty$, the quantity
$$\sup_{g \in \mathfrak{S}}{\int_{0}^{1}g(t)|f(t)| dt}$$
is finite and attained by some $g_0 \in \mathfrak{S}$.
\end{theorem}

\begin{proof}
Let $\mathcal{A}$ be the subset of $\mathcal{O}(\D)$ consisting of functions whose coefficients are either \( 0 \) or \( 1 \) and with the first coefficient equal to zero.
$$\mathcal{A}=\left\{ \sum_{n=1}^\infty{ a_n z^n}\in\mathcal{O}(\D) \ : \ a_i\in\{0,1\} \ \forall i\in\N \right\}.$$
Since the topology of uniform convergence on compact sets is metrizable, and since that $\mathcal{A}$ is sequentially compact by the diagonalization principle, it follows that $\mathcal{A}$ is compact. 

The following equivalence holds.
$$S\in \mathcal{S} \Leftrightarrow f_S \in \mathfrak{S}:=\mathcal{A} \cap G^{-1}(\mathcal{A}),$$
where 
\begin{align*}
G \ : \  \mathcal{O}(\D) & \longrightarrow \mathcal{O}(\D) \\  f & \longmapsto \left(z \mapsto \frac{f^2(z)+f(z^2)}{2}\right) .
\end{align*}
Indeed, if $B\subset\N$ and $z \in \D$, we have
$$ f_B(z)^2 = \sum_{n=0}^\infty{z^n \sum_{k=0}^n{\mathbbm{1}_B(k)\mathbbm{1}_B(n-k)}},$$
so
$$\frac{f_B(z)^2+f_B(z^2)}{2} = \sum_{n=0}^\infty{z^n \frac{\sum_{k=0}^n{\mathbbm{1}_B(k)\mathbbm{1}_B(n-k)}+\mathbbm{1}_{n \equiv 0[2]}\mathbbm{1}_B(n/2)}{2}}.$$
The general term of this power series corresponds exactly to the number of representations of \( n \) as a sum of integers \( \{b_i + b_j\} \) from \( B \), with \( b_i \leq b_j \). Therefore, the Sidon condition is precisely the requirement that all its coefficients are either \( 0 \) or \( 1 \). And as $\mathcal{S}\subset \mathcal{P}(\N^*)$, the first coefficient is equal to zero.

Therefore $S\in \mathcal{S}$ if and only if $f_S \in \mathfrak{S}$. But \( G \) is continuous, since both \( f \mapsto f^2 \) and \( f \mapsto (z \mapsto f(z^2)) \) are continuous. We can then deduce that \( G^{-1}(\mathcal{A}) \) is closed, as the preimage of a closed set under a continuous map. Furthermore, since \( \mathcal{A} \) is compact, \( \mathfrak{S} \) is a closed subset of a compact set, and therefore it is compact. This proves the first part of the theorem.

Then, let $g\in \mathfrak S$, we have $G(g) \in \mathcal{A}$ and so $\forall t \in ]0,1[$
$$ G(g)(t) \leq \sum_{n=1}^{\infty}t^n=\frac{t}{1-t},$$
and so
$$g(t)^2 \leq \frac{2t}{1-t}-g(t^2)\leq \frac{2t}{1-t}, $$
and finally
$$g(t) \leq \frac{\sqrt{2t}}{\sqrt{1-t}}. $$
Let $f\in \mathcal{C}^{0}([0,1[)$ such that $\int_{0}^1{\frac{|f(t)|\sqrt t}{\sqrt{1-t}}}dt < +\infty$ and define
$$m(f) = \sqrt{2}\int_{0}^1\frac{|f(t)|\sqrt t}{\sqrt{1-t}}dt.$$
Thus we have
$$\int_{0}^1g(t)|f(t)|dt \leq m(f).$$
The Dominated Convergence Theorem ensures that $g \mapsto \int_{0}^1g(t)|f(t)| dt$ is sequentially continuous, and thus continuous on \( \mathfrak{S} \) with the topology induced from \( \mathcal{O}(\D) \). This concludes the proof, since every continuous function on a compact set attains its extremes.
\end{proof}

As a corollary of Theorem \ref{mainResult}, we establish the following result.

\begin{theorem}\label{MainResult}
    Let $\alpha> \frac{1}{2}$, there exists a Sidon set $S_\alpha\subset\N^*$ such that 
    $$\sum_{s\in S_\alpha}\frac{1}{s^\alpha}=\sup\left\{ \sum_{s\in S}\frac{1}{s^\alpha} \ : \ S\subset\N^* \text{ is a Sidon set}\right\}.$$
\end{theorem}

\begin{proof}
    Let $\alpha> \frac{1}{2}$. For all $g\in\mathcal{A}$, let $S_g$ the set of indices of the nonzero coefficients of $g$ (in its analytical development). We aim to use Theorem \ref{mainResult}, and we will therefore show that there exists a function $f_{\alpha}$ such that $\int_{0}^1{\frac{|f(t)|\sqrt t}{\sqrt{1-t}}}dt < +\infty$ and for all $ g \in \mathfrak{S}$
$$ \sum_{s\in S_g}s^{-\alpha} = \int_{0}^1{g(t)f_{\alpha}(t)dt}.$$
Let $g$ therefore be in $\mathfrak S$. Recall that $\Gamma(z)=\int_{0}^{+\infty}{e^{-t}t^{z-1}dt}$.
By the substitution $st=u$, we have $\int_{0}^{+\infty}{e^{-st}t^{\alpha-1}dt}=s^{-\alpha}\Gamma(\alpha)$ and by writing $S_g=(s_n)_{n\geq 1}$,
\begin{align*}
 \sum_{s\in S_g}s^{-\alpha} & = \frac{1}{\Gamma(\alpha)}\sum_{n \in \mathbb{N}} \int_{0}^{+\infty}{e^{-s_nt}t^{\alpha-1}dt} \\
    & = \frac{1}{\Gamma(\alpha)}\int_{0}^{+\infty}\sum_{n \in \mathbb{N}} {e^{-s_nt}t^{\alpha-1}dt},
\end{align*}
by Fubini-Tonelli theorem. Thus
$$\sum_{s\in S_g}s^{-\alpha}=\frac{1}{\Gamma(\alpha)}\int_{0}^{+\infty}{f_{S_g}(e^{-t})t^{\alpha-1}dt} , $$
and by the substitution $u=e^{-t}$
$$\sum_{s\in S_g}s^{-\alpha}=\frac{1}{\Gamma(\alpha)}\int_{0}^{1}{f_{S_g}(u)\frac{(-\ln(u))^{\alpha-1}}{u}du}.$$

Now as $\alpha>1/2$, $f_\alpha :u \mapsto \frac{(-\ln(u))^{\alpha-1}}{u}$ verifies $\int_{0}^1{\frac{|f_\alpha(t)|\sqrt t}{\sqrt{1-t}}}dt < +\infty$. Indeed, the integrand is bounded near $t=1$, and near $t=0$, this follows from the integrability of $x^{\alpha-1}e^{-x/2}$ at $+\infty$ after the substitution $t=e^{-x}$. So by Theorem \ref{mainResult}, there exists $g_\alpha\in\mathfrak S$ such that 
$$\int_{0}^{1}g_\alpha(t)|f_\alpha(t)| dt=\sup_{g \in \mathfrak{S}}{\int_{0}^{1}g(t)|f_\alpha(t)| dt}=\sup_{g \in \mathfrak{S}}\sum_{s\in S_g}s^{-\alpha}.$$
It remains only to define $S_\alpha=S_{g_\alpha}$ to complete the proof.
\end{proof}
\begin{rmq}
    The bound $\alpha>1/2$ is optimal, as we'll see in Corollary \ref{alpha<1/2}.
  
\end{rmq}

\vspace{0.2cm}

For $\alpha=1$, Theorem \ref{MainResult} shows the existence of a Sidon set whose reciprocal sum is equal to the DDC.

\begin{theorem}
    There exists a Sidon set $S\subset\N$ such that 
    $$\sum_{s\in S}\frac{1}{s}=\mathrm{DDC}.$$
\end{theorem}

\section{Properties of Sidon sets with maximum reciprocal sum}\label{sectionProprietes}

Theorem \ref{MainResult} shows the existence of certain maximal Sidon sets. In this section, we will focus on some of their properties. First of all, it is important to note that we do not know whether there is uniqueness.

\begin{question}
    Are there several Sidon sets with a reciprocal sum equal to the constant DDC?
\end{question}

Throughout this section, $S$ will denote a Sidon set with maximum reciprocal sum: $\max_{S' \in \mathcal{S}}{\sum_{s \in S'}{\frac{1}{s}}} =\sum_{s \in S}{\frac{1}{s}}$. Theorem \ref{MainResult} guarantees its existence. We begin with some elementary remarks that are immediately implied by the maximality. First, $S$ is maximal for inclusion, which means that it is impossible to add an element to $S$. In other words :
$$S+S-S = \mathbb{N}.$$
 Furthermore, let $A=\{s_i\}_{i \in I}$ denote some elements of $S$, and $B:=\{s'_j\}_{j\in J}$ some integers that are not in $S$ such that $\sum_{j \in J}{\frac{1}{s'_j}} > \sum_{i \in I}{\frac{1}{s_i}}$ where the sum on the right might be $+\infty$. Then, $(S \backslash  A) \cup B$ is not a Sidon set. This follows also immediately from the maximality.

Since Zhang \cite{Zhang} proved that the Mian-Chowla Sidon set is not maximal for this linear form, it seems intuitive to think that $|S\cap\llbracket 1,n \rrbracket|$ would have to grow at least like $n^{1/3}$.
We were unable to prove this result. However, it is possible to show the following asymptotic estimates.

\begin{theorem} 
Let $S$ be a Sidon set such that $\max_{S' \in \mathcal{S}}{\sum_{s \in S'}{\frac{1}{s}}} =\sum_{s \in S}{\frac{1}{s}}$.
   There exists $C>0.257$, such that $|S\cap \llbracket 1,n \rrbracket| > Cn^{1/4}$ for all $n \in \mathbb{N}^*$. Moreover 
$$\limsup {\frac{|S\cap \llbracket 1,n \rrbracket|}{n^{1/3}}} > 0.$$
\end{theorem}

\begin{proof}
   Let $n \in \mathbb{N}^*$ and $\{s_k\}_{k \in \mathbb{N}}$ the elements of $S$ that are greater than $n$. Since $S$ is as Sidon set, for all $k \geq 0$, we have
$$s_k \geq n+\frac{k(k+1)}{2} .$$
Which gives
$$ \sum_{k=0}^{\infty}{\frac{1}{s_k}} \leq   \sum_{k=0}^{\infty}\frac{1}{n+\frac{k(k+1)}{2}}\leq \int_{0}^{\infty}\frac{dt}{n+\frac{t^2}{2}}
  \leq  \frac{1}{n}\int_{0}^{\infty}\frac{dt}{1+\frac{t^2}{2n}}\leq  \frac{\pi}{\sqrt{2n}},$$
that is
\begin{equation}\label{majRestes} \nonumber
    \sum_{s \in S, s \geq n}{\frac{1}{s}} \leq \frac{\pi}{\sqrt{2n}}.
\end{equation}

Now let us remark the following fact. Let $n \in \mathbb{N}^*$, $S_n=S\cap \llbracket 1,n \rrbracket$ and $\alpha, c>0$.
We define Sidon sets $\tilde{S}_k$ by finite induction and $S'$ as follows : 
\begin{itemize}
\item {$\tilde{S}_0 = S_n$.} 
\item{Suppose $\tilde{S}_k$ is defined. If $A_k=\{j \leq c n^{\alpha} \ : \ \tilde{S}_k \cup \{j\} \text{ \ is a Sidon set}\}$ is non empty, then $\tilde{S}_{k+1} = \tilde{S}_k \cup \{\min A\}$. If $A$ is empty, the recursion ends, and we set $S' = \tilde{S}_k \setminus S$}.
\end{itemize}

$S'$ satisfies the following properties. 
\begin{enumerate}
\item{$S'\subset \llbracket 1,cn^\alpha \rrbracket$}
\item{$\llbracket 1,cn^\alpha \rrbracket \subset (S_n\cup S')+(S_n\cup S')-(S_n\cup S'$).}
\end{enumerate}
Indeed, 1 is obvious and 2 is because if it's not true, it would contradict the fact that the algorithm ended, because at least one more element could have been added to $S'$. It follows from 2 that :

\begin{equation}\label{minSn} \nonumber
    |S_n|+|S'| \geq c^{1/3}n^{\alpha/3}.
\end{equation}
But since $S \cup S'$ is a Sidon set, we have
$$\sum_{s \in S_n \cup S'}{\frac{1}{s}} \leq \sum_{s \in S}{\frac{1}{s}}, $$
and as $S_n\cap S'=\varnothing$
$$\sum_{s \in S'}{\frac{1}{s}} \leq \sum_{s \in S, s \geq n}{\frac{1}{s}},$$
so by \eqref{majRestes}
\begin{equation}\label{haut} \nonumber
    \sum_{s \in S'}{\frac{1}{s}} \leq \frac{\pi}{\sqrt{2n}}. 
\end{equation}
But on the other hand
$$\sum_{s \in S'}{\frac{1}{s}}\geq \sum_{k=\max S'-|S'|}^{\max S'}{\frac{1}{s}}\geq \int_{cn^\alpha-|S'|}^{cn^\alpha}\frac1t \text{d}t\geq -\ln\left( 1-\frac{|S'|}{cn^\alpha}\right)\geq \frac{|S'|}{cn^\alpha}.$$
Combined with \eqref{haut}, this gives
$$|S'| \leq \frac{\pi}{\sqrt{2}} c n^{\alpha-1/2}, $$
and so by \eqref{minSn}
$$ |S_n| \geq c^{1/3}n^{\alpha/3}-\frac{\pi}{\sqrt{2}} c n^{\alpha-1/2}.$$
In order for the $\alpha/3$ exponent to be dominant, one has to choose $\alpha$ such that $\alpha/3 \geq \alpha-1/2$, ie $\alpha \leq 3/4$. We thus choose $\alpha=3/4$ and $c= \frac{2^{3/4}}{(3\pi)^{3/2}}$ to obtain
$$|S_n| \geq C n^{1/4} , $$
with $C> 0.257$ for all $n\in\N^*$.

Let us now prove that $\limsup{\frac{|S(n)|}{n^{1/3}}}>0$. More precisely, we will show that this limsup is greater than $6^{-1/3}$. Let $d<6^{-1/3}$ , $N \in \mathbb{N^*}$ and suppose that $|S(n)| \leq dn^{1/3}$ for all $n\geq N$. Let's build another Sidon set $\tilde{S}$ using a greedy algorithm :
\begin{itemize}
    \item Initialise $\tilde{S}_0 = S_N$.
    \item Suppose $\tilde{S}_k$ is constructed. We now take $\tilde{s}_k = \min \{p \in \mathbb{N}^*\setminus S_N :  \{p\} \cup \tilde{S}_k \ \text{is a Sidon set} \}$, and set $\tilde{S}_{k+1} = \tilde{S}_k \cup \{\tilde{s}_k\}$.
\end{itemize}
Finally, set $\tilde{S} = \bigcup_{k\in \mathbb{N}}{\tilde{S_k}}$. Since 
\begin{eqnarray} \nonumber
|\tilde{S}_k+\tilde{S}_k-\tilde{S}_k| \leq |\tilde{S}_k|^3,
\end{eqnarray}
we cannot have $\tilde{S}_k+\tilde{S}_k-\tilde{S}_k = \llbracket 1,|\tilde{S}_k|^3+1\rrbracket$ and thus we must have 
$$\tilde{s}_{k+1} \leq |\tilde{S}_k|^3+1 \leq (k+|S_N|)^3+1, $$
for all $k\in\mathbb{N}$. Now, by definition of $S$, we have
$$\sum_{s \in S}{\frac{1}{s}}\geq \sum_{s \in \tilde{S}}{\frac{1}{s}}, $$
and therefore by simplifying the elements of $S_N$, we get
$$  \sum\limits_{\substack{s \in S\\ s > N}}{\frac{1}{s}} \geq \sum_{k=1}^{+\infty} {\frac{1}{(k+|S_N|)^3+1}}\geq \sum_{k=1+\lfloor d N^{1/3}{\rfloor}}^{+\infty} {\frac{1}{k^3+1}}\nonumber \geq \int_{2+\lfloor d N^{1/3}{\rfloor}}^{+\infty}{\frac{dt}{t^3}} \nonumber \geq {\frac{1}{(4+2d N^{1/3})^2}}  .$$
But on the other hand, by Abel summation 
\begin{align*}
     \sum\limits_{\substack{s \in S\\ s > N}}{\frac{1}{s}}  &=\sum_{k=N+1} {\frac{|S(k)|-|S(N)|}{k(k+1)}}\leq\sum_{k=N+1}^{+\infty} {\frac{|S(k)|}{k(k+1)}}\leq d\sum_{k=N+1}^{+\infty} {\frac{k^{1/3}}{k(k+1)}} \\ 
& \leq d\int_{N}^{+\infty}{\frac{dt}{t^{5/3}}}\leq \frac{3d}{2N^{2/3}}.
\end{align*}
Finally we have $\frac{1}{4d^2N^{2/3}}(1-o(1)) \leq \frac{3d}{2N^{2/3}}$ which gives $6^{-1/3}(1-o(1)) \leq d$ and leads to absurdity. This concludes the proof.
\end{proof}

\begin{rmq}
    Here we have only given properties of Sidon sets maximizing $\sum\frac{1}{s}$, but in Section \ref{sectionMainResult} we showed the existence of Sidon sets maximizing many different functionals. Several pieces of information can be extracted from these Sidon sets by playing on the choice of the functional they maximize. Unfortunately, we were unable to extract sufficiently good density information to improve Ruzsa's result \cite{Ruzsa}.
\end{rmq}

\section{Integrability of functions in $\mathfrak{S}$}

In Section \ref{sectionMainResult}, we prove Theorem \ref{MainResult} from Theorem \ref{mainResult} by expressing the mappings $\sum_{s\in S}s^{-\alpha}$ in terms of scalar products on $\mathfrak{S}$. It is therefore legitimate to be interested in the integrability of functions in $\mathfrak{S}$. 

\begin{theorem} Let $\alpha<2$, we have $\mathfrak S \subset \mathcal{L}^{\alpha}(]0,1[)$, i.e. 
$$\int_{0}^{1}f(t)^{\alpha} \text{d}t < +\infty,$$
for all $f\in\mathfrak S$.
\end{theorem}

\begin{proof}
Let $ \alpha <2$, $f \in \mathfrak{S}$ and $t\in ]0,1[$, we have 
$$0 \leq f(t) \leq \frac{2}{\sqrt{1-t}},$$
so
$$0 \leq f(t)^{\alpha} \leq \frac{2}{(1-t)^{\alpha/2}} ,$$
which justifies the integrability of $f^{\alpha}$.
\end{proof}
Unfortunately, this result cannot be extended to $\alpha=2$.

\begin{theorem}\label{pasL2}
There is $S\in\mathcal S$ such that
$$ \int_{0}^{1}{f_S(t)^2 \text{d}t} = +\infty.$$
In other words, $\mathfrak{S} \not \subset \mathcal{L}^2(]0,1[)$.
\end{theorem}
The proof of this theorem is based on the following lemma.

\begin{lem}\label{div}
    Let $A \subset \mathbb{N}$ such that $\limsup \frac{|A\cap \llbracket1,n\rrbracket|}{n} >0$, then
$$\sum_{a \in A}{\frac{1}{a}} = + \infty.$$
\end{lem}
\begin{proof}
Let $A \subset \mathbb{N}$, and let us denote  
$$\alpha  = \limsup \frac{|A \cap\llbracket1,n \rrbracket|}{n}.$$
By definition of  $\alpha$, there is a sequence $(n_k)_{k\in\mathbb{N}}$ such that $|A \cap \llbracket 1,n_k \rrbracket| \geq \frac{\alpha}{2} n_k$ for all $k \in \mathbb{N}$. By an Abel transformation, one has 
$$\sum_{a \in A}{\frac{1}{a}}=\sum_{k=1}^{n}\frac{\mathbbm{1}_{A}(k)}{k} = \sum_{k=1}^{n}\frac{|A\cap \llbracket 1,k\rrbracket|}{k(k+1)}+ \frac{A(n)}{n}, $$
but
$$\sum_{l=(1-\alpha/4)n_k}^{n_k}\frac{|A\cap \llbracket 1,l\rrbracket|}{l(l+1)} \geq \sum_{l=(1-\alpha/2)n_k}^{n_k}\frac{n_k}{2l(l+1)} \geq \alpha/4 .$$
Since these Cauchy slices don't converge to 0, the series can't converge, and since all of its terms are positive, it diverges.
\end{proof}
We can now prove Theorem \ref{pasL2}.
\begin{proof}
   We will in fact prove that every Sidon set $S$ that achieves $\limsup{\frac{S(n)}{\sqrt{n}}}>0$ satisfies
   $$\int_{0}^{1}{f_S(t)^2 dt} = +\infty.$$
Let $S \in \mathcal{S}$ be such a Sidon set (see \cite{HalberstamRoth} p.89  for proof of their existence), we have
$$ \limsup \frac{S(n)^2+S(n)}{2n} >0,$$
so 
$$\limsup \frac{(S+S)(n)}{n} >0$$
thus $ \sum_{\sigma \in S+S} \frac{1}{\sigma} = +\infty$ by Lemma \ref{div}. But by Fubini's theorem
$$ \int_{0}^{1} {f_{S+S}(t)dt} =\int_{0}^{1} \sum_{\sigma \in S+S} {t^{\sigma}dt}=\sum_{\sigma \in S+S}\frac{1}{\sigma+1}\geq \sum_{\sigma \in S+S}\frac{1}{\sigma}-1=+\infty,$$
where the inequality holds because, writing $S+S=(\sigma_n)_{n\in\N^*}$, we have $\sigma_n+1\leq \sigma_{n+1}$ and $\sigma_1\geq 1$.
Finally we have shown (see the proof of theorem \ref{mainResult}) that for all $t \in ]0,1[$
$$  f_{S+S}(t) = \frac{f_S(t)^2+f_S(t^2)}{2}.$$
Therefore
$$  f_{S+S}(t) \leq \frac{f_S(t)^2+f_S(t)}{2},$$
and by positivity
$$\int_{0}^{1} {f_{S}^2(t)dt}\geq 2\int_{0}^{1} {f_{S+S}(t)dt}-\int_{0}^{1} {f_{S}(t)dt}\geq 2\int_{0}^{1} {f_{S+S}(t)dt}-\mathrm{DDC}=+\infty.$$
\end{proof}
We will now prove the following corollary, which extends Theorem \ref{MainResult} to all $\alpha\leq 1/2$.
\begin{cor}\label{alpha<1/2}
    Let $\alpha\leq 1/2$, there exists an infinite Sidon set $S$ such that
    $$\sum_{s \in S} {\frac{1}{s^\alpha}} = + \infty.$$
\end{cor}

\begin{proof}
    Trivially, we have $\sum_{s \in S} {\frac{1}{s^\alpha}}\geq \sum_{s \in S} {\frac{1}{\sqrt{s}}}$ for all $\alpha\leq 1/2$, so we just need to prove the result for $\alpha=1/2$. By Theorem \ref{pasL2}, there is $S\in\mathcal S$ such that
$$ \int_{0}^{1}{f_S(t)^2 \text{d}t} = +\infty.$$
But for all $t\in ]0,1[$ we have $f_S(t)\leq \frac{\sqrt{2}}{\sqrt{1-t}}$ (see the proof of Theorem \ref{mainResult}). Thus
$$\int_{0}^{1}{f_S(t)^2 \text{d}t}\leq \sqrt{2} \int_{0}^{1}{\frac{f_S(t)}{\sqrt{1-t}} \text{d}t}=\sqrt{2} \int_{0}^{1}{\sum_{s\in S}\frac{t^s}{\sqrt{1-t}} \text{d}t}=\sqrt{2} \sum_{s\in S}\int_{0}^{1}\frac{t^s}{\sqrt{1-t}} \text{d}t,$$
by Fubini's theorem and positivity. So $\sum_{s\in S}\int_{0}^{1}\frac{t^s}{\sqrt{1-t}} \text{d}t$ is a divergent series. On the other hand, the substitution $t = \sin^2(x)$, gives
$$\int_{0}^{1}\frac{t^s}{\sqrt{1-t}} \text{d}t=2\int_{0}^{\pi/2}\sin(x)^{2s+1} \text{d}t\sim \frac{\sqrt\pi}{2\sqrt{s}},$$
by a classical result on Wallis integrals. So finally $\sum_{s \in S} {\frac{1}{\sqrt s}} = + \infty$, which concludes the proof.
\end{proof}

\section{Approximation by finite Sidon sets}

Another consequence of Theorem \ref{mainResult} is the following: with respect to the topology of \( \mathcal{O}(\D) \), the subset of \( \mathfrak{S} \) consisting of the generating functions of finite sets is dense. For each \( n \in \mathbb{N}^* \), let \( \mathfrak{S}_n \) denote the set of generating functions of the Sidon sets of \( \llbracket 1, n \rrbracket \). We thus obtain the following result.

\begin{theorem}
Let \( \phi \) be a continuous function on \( \mathfrak{S} \) endowed with the topology induced from the topology of \( \mathcal{O}(\D) \), taking values in \( \mathbb{R}^+ \).
Then
\[
\max_{f \in \mathfrak{S}} \phi(f) = \lim_{n \to \infty} \max_{f \in \mathfrak{S}_n} \phi(f).
\]
\end{theorem}

\begin{proof}
    Since \( \mathfrak{S}_n \subset \mathfrak{S} \), we have
\[
\forall n \in \mathbb{N}^*, \quad \max_{f \in \mathfrak{S}_n} \phi(f) \leq \max_{f \in \mathfrak{S}} \phi(f).
\]
On the other hand, let \( S_\phi \) be a Sidon set such that:
\[
\phi(f_{S_{\phi}}) = \max_{f \in \mathfrak{S}} \phi(f),
\]
which exists by Theorem \ref{mainResult}. Additionally, let \( S_{\phi,n} = S_{\phi} \cap \llbracket 1, n \rrbracket \). We then have
\[
\phi(f_{S_{\phi,n}}) \leq \max_{f \in \mathfrak{S}_n} \phi(f),
\]
and by the continuity of \( \phi \)
\[
\phi(f_{S_{\phi}}) = \lim_{n \to \infty} \phi(f_{S_{\phi,n}}).
\]
This concludes the proof.
\end{proof}

\section{New bounds for the DDC}\label{sectionMajDDC}

In this section, we show the following result.

\begin{theorem}
    We have $2.16150003<\mathrm{DDC}\leq 2.247307$.
\end{theorem}

For the upper bound, we adopt the method developed by Taylor \cite{Taylor}, who proposes dividing the series into three blocks. The first block, corresponding to the small indices, is bounded by the best sum we have for a finite Sidon set. The second block, corresponding to the middle indices, is bounded by a result established in \cite{Taylor2}. Finally, we bound the third block, which corresponds to remainder of the series. It is the bound on this last block that we aim to improve. Taylor cuts the first block after 20 terms, but any knowledge about the maximal reciprocal sum of Sidon sets of cardinality greater than 20 will further improve the result.

To bound the third block, Taylor use 
$$
\sum_{n=1100}^{\infty}1/s_n\leq\sum_{n=1100}^{\infty} \frac{1}{c(n-1) n / 2}=\frac{2 / c}{1100-1}=\frac{2 / c}{1099},
$$
where $c$ is a constant less than $1.9$ and $(s_n)$ is a Sidon set. Using Lindström's inequality \cite{Lindstrom}, we improve this step. This is the purpose of the following lemma.

\begin{lem}
    If $S=(s_n)$ is a Sidon set and $N\in\N^*$, we have 
    $$\sum_{n=N}^{\infty}1/s_n<\frac{2}{N-\sqrt{N}},$$
    and more specifically 
    $$\sum_{n=N}^{\infty}1/s_n<2\ln\left(1-\frac{1}{\sqrt{N}}\right)+\frac{2}{\sqrt{N}-1}.$$
\end{lem}

\begin{proof}

Lindström \cite{Lindstrom} showed that the cardinality of a Sidon set in the $n$ first integers is less than $\sqrt{n}+n^{1/4}+1/2$. So the cardinality of a Sidon set in the $\lfloor (n-\sqrt{n})^2\rfloor$ first integers is less than $n-1$. Thus the $n$th element of a Sidon set is necessarily strictly greater than $(n-\sqrt{n})^2$ and so we have
$$
\sum_{n=N}^{\infty}1/s_n\leq\sum_{n=N}^{\infty} \frac{1}{(n-\sqrt{n})^2}<\int_{N}^\infty \frac{1}{(x-\sqrt{x})^2}dx=2\int_{\sqrt{N}}^\infty \frac{1}{y(y-1)^2}dy.
$$
Then
\begin{align*}
    \int_{\sqrt{N}}^\infty \frac{1}{y(y-1)^2}dy & =\lim_{A\rightarrow \infty}\int_{\sqrt{N}}^A \frac{1}{x}dx-\int_{\sqrt{N}}^A \frac{1}{x-1}dx+\int_{\sqrt{N}}^A \frac{1}{(x-1)^2}dx \\ & =\ln((\sqrt{N}-1)/\sqrt{N})+1/(\sqrt{N}-1) \\ & \leqslant \frac{-1}{\sqrt{N}}+\frac{1}{\sqrt{N}-1}=\frac{1}{N-\sqrt{N}}.
\end{align*}
The last two lines prove the theorem.
\end{proof}
Applying this lemma to $N=1100$, we find
$$
\sum_{n=1100}^{\infty}1/s_n<0.000947,
$$
and finally $\mathrm{DDC}\leq 2.247307$, which proves the upper bound of Theorem \ref{majDDC}. The lower bound, as explained in the introduction, is obtained by completing $K$, the Sidon set of 1010 elements discovered by L. Kleinwaks, using the greedy algorithm. Since $\max K=13655199=K_m$, $K\cup\left\lbrace 2^kK_m \ : \ k\geqslant 1\right\rbrace$ is a Sidon set, and we have
$$\sum_{k\geqslant 1}\frac{1}{2^kK_m}=\frac{1}{K_m}>7.3\times 10^{-8} .$$
Therefore, $\mathrm{DDC}>S_K+7.3\times 10^{-8}>2.16150003$ which concludes the proof of Theorem \ref{majDDC}.

\begin{rmq}
     In a recent breakthough, Balogh, Füredi, and Roy \cite{BaloghFurediRoy} showed that the cardinality of a Sidon set in the $n$ first integers is less than $\sqrt{n}+0.998n^{1/4}+O(1)$. This was further improved by Carter, Hunter, and O’Bryant \cite{CarterHunterOBryant}, which showed that is less than $\sqrt{n}+0.98183n^{1/4}+O(1)$. These bounds are more precise than Lindström's, but the $O(1)$ terms would need to be made explicit in order to be used here.
\end{rmq}

\section{Other patterns}\label{SectionOtherPatterns}

In this section, we show how our method can be adapted to other constraint equations. We'll use the same notations as in Section \ref{sectionMainResult}. Recall that for all $B\subset\mathbb{N}$, we denote
$$f_B(z)=\sum_{n=0}^\infty{\mathbbm{1}_B(n)z^n},$$
and
$$\mathcal{A}=\left\{ \sum_{n=0}^\infty{ a_n z^n}\in\mathcal{O}(\D) \ : \ a_i\in\{0,1\} \ \forall i\in\N \right\}.$$

\subsection{$B_2[g]$ sequences}

First of all, we can extend Theorem \ref{MainResult} to $B_{2}[g]$. 

\begin{theorem}
Let $g\geq 2$ and $\alpha> \frac{1}{2}$, there exists $B_\alpha\in B_2[g]$ such that
    $$\sum_{b\in B_\alpha}\frac{1}{b^\alpha}=\sup\left\{ \sum_{b\in B} b^{-\alpha} \ : \ B\in B_2[g]\right\}.$$
\end{theorem}

\begin{proof}
Let $g\geq 2$. Following the same idea as for the proof of Theorem \ref{mainResult}, we define
$$\mathcal{A}_g=\left\{ \sum_{n=0}^\infty{ a_n z^n}\in\mathcal{O}(\D) \ : \ a_i\in\{0,\dots,2g\} \ \forall i\in\N \right\},$$
and 
\begin{align*}
G \ : \  \mathcal{O}(\D) & \longrightarrow \mathcal{O}(\D) \\  f & \longmapsto \left(z \mapsto f^2(z)+f(z)^2\right).
\end{align*}
Similarly, one shows that $\mathfrak S_{g}=\mathcal{A} \cap G^{-1}(\mathcal{A}_g)$ is compact and that $S\in\mathcal B_2[g]$ if and only if $f_S\in\mathfrak S_{g}$.

Furthermore, let $f\in \mathfrak S_{g}$, we have $G(f) \in \mathcal{A}_g$ and so for all $t \in ]0,1[$
$$ G(f)(t) \leq \sum_{n=1}^{\infty}2gt^n=\frac{2gt}{1-t}.$$
But $G(f)(t)\geq f(t)^2$, so
$$f(t) \leq C_{g}\left(\frac{t}{1-t}\right)^{1/2}, $$
where $C_{g}\in\R$. We conclude in the same way as for Theorems \ref{mainResult} and \ref{MainResult}.
\end{proof}


Our method can be adapted to many other linear constraints. We give an example below.

\subsection{Sum-free sets}

A set is said to be sum-free if it has no solution to the equation $x+y=z$. Let $\mathcal{F}$ the set of the free-sum sets in $\N^*$.

\begin{theorem}
    Let $\alpha\in\R$. There exists $F_\alpha\in\mathcal F$ such that
    $$\sum_{f\in F_\alpha} \frac{1}{f^\alpha}=\sup\left\lbrace \sum_{f\in F} \frac{1}{f^\alpha} \ : \ F\in\mathcal F\right\rbrace.$$
\end{theorem}

\begin{proof}
    The set of odd integers is trivially a sum-free set and $\sum_{n\in\N}\frac{1}{(2n+1)^{\alpha}}=+\infty$ for all $\alpha\leq 1$. All that remains is to show the result for $\alpha>1$.

    This time, we define
$$\mathcal{A}'=\left\{ \sum_{n=0}^\infty{ a_n z^n}\in\mathcal{O}(\D) \ : \ a_i\in\{0,...,i+1\} \ \forall i\in\N \right\},$$
and 
\begin{align*}
H \ : \  \mathcal{O}(\D) & \longrightarrow \mathcal{O}(\D) \\  f & \longmapsto f^2-(zf)' .
\end{align*}
Indeed,
    $$f_F(z)^2+(zf_F(z))'=\sum_{n\geq 0}z^n \underbrace{\left( \sum_{k=0}^n \mathbbm{1}_F(k)\mathbbm{1}_F(n-k)+(n+1)\mathbbm{1}_F(n)  \right)}_{c_n}, $$
and so $F\in\mathcal F$ if and only if $f_F\in\mathcal A$ and $c_n\in\{ 0,\dots,n+1\}$. Now $\mathcal A'$ is compact, so $F\in\mathcal F$ if and only if $f_F$ is in the compact $\mathcal A\cap H^{-1}(\mathcal A')$. Finally, of course, for $t\in ]0,1[$, we have
$$f_F(t)\leq \frac{t}{1-t},$$
and $\int_{0}^1{\frac{|f_\alpha(t)| t}{1-t}}dt < +\infty$ for $f_\alpha :u \mapsto \frac{(-\ln(u))^{\alpha-1}}{u}$ and $\alpha>1$, so we can conclude as for Theorem \ref{MainResult}.
\end{proof}

Since $\sum_{n\in\N} \mathbbm{1}_E(n)n^{-\alpha}$ is continuous on $\mathcal P(\N)$ (with the discrete product topology) for all $\alpha>1$, this last example does not require the use of our method. In fact, whatever the original pattern, if we know that the application to be maximized is continuous, we have a more direct existence result, Theorem \ref{allPatterns} below, which is fairly elementary but does not seem to be well known in the literature, since the existence of sets maximizing certain functions is often said to be still unknown. 

For instance, building on earlier work of Erd\H{o}s \cite{ErdosSumfree}, Levine and O'Sullivan \cite{LevineSullivan} investigated the reciprocal sums of strongly sum-free sets, that is, sets in which no element can be expressed as a finite sum of the remaining elements. They introduced $\lambda$ as the supremum of the reciprocal sums taken over all strongly sum-free sets, and established both upper and lower bounds for this quantity. In their concluding remarks, they state that the primary problem is that of determining the precise value of $\lambda$ and those strongly sum-free sets that attain this supremum 'if they exist' (what this article shows). Although subsequent work aimed at improving the bounds no longer explicitly mentions this issue, it nevertheless retains the use of the supremum rather than the maximum (see \cite{AbbottFreeSum, ChenFreeSum, ZhangFreeSum}).

\subsection{All patterns}

We provide $\mathcal P(\N)$ with the discrete product topology (with $\mathcal P(\N)$ seen as $\{0,1\}^\N$). Our next result concerns all closed parts of $\mathcal P(\N)$, which covers almost all the usual patterns. For example, $B_h[g]$ is a closed part. Indeed, if $(B_k)_{k\in\N}\in B_h[g]^\N$ tends to $B\notin B_h[g]$, then there exist $k\in\N$ and $(b_1^{(i)},\dots,b_h^{(i)})\in B^h$, $i\in \{1,\dots, g+1\}$ such that $b_1^{(i)}+\dots+b_h^{(i)}=k$ for all $i$. But as $(B_k)$ tends towards $B$, there exists a rank $N$ such that $b_j^{(i)}\in B_N$ for all $j,i$, which is absurd.

\begin{theorem}\label{allPatterns}
    Let $\mathcal{A}$ be a closed part of $\mathcal P(\N)$ and $\mathcal{F}:\mathcal A \rightarrow \R$ continuous, then there exists $A\in\mathcal A$ such that $\mathcal F(A)=\sup \mathcal{F}(\mathcal A)$.
\end{theorem}
\begin{proof}
    This proof is very elementary and relies essentially on the fact that $\mathcal P(\N)$ is sequentially compact for the product discrete topology. It's a well-known fact, and once again results from the diagonal process. The rest is a textbook case.

    Let $(A_k)\in\mathcal A^\N$ such that $\mathcal F(A_k)\rightarrow\sup \mathcal{F}(\mathcal A)$. As $\mathcal P(\N)$ is sequentially compact, there is $A\in\mathcal P(\N)$ such that $A_{n_k}\rightarrow A$. But $A\in \mathcal A$ because $\mathcal A$ is close. Finally $\mathcal F$ is continous so $\mathcal F(A)=\lim \mathcal F(A_{n_k})=\sup \mathcal{F}(\mathcal A)$.
\end{proof}

The difficulty in using Theorem \ref{allPatterns} is not in showing that $\mathcal A$ is closed but rather that $\mathcal F$ is continuous, hence the importance of the method in Section \ref{sectionMainResult}. Let's give an example of the application of Theorem \ref{allPatterns}.

\begin{cor}\label{Bhg}
Let $g\geq 2$, $h \geq 2$, and $\alpha> \frac{1}{h}$, there exists $B_\alpha\in B_h[g]$ such that
    $$\sum_{b\in B_\alpha}\frac{1}{b^\alpha}=\sup\left\{ \sum_{b\in B} \frac{1}{b^\alpha} \ : \ B\in B_h[g]\right\}<+\infty.$$
\end{cor}

\begin{proof}
    This is a direct application of Theorem \ref{allPatterns}, since we saw at the start of this subsection that $B_h[g]$ is closed. All that remains is to show that for all $\alpha>1/h$, the function
 \begin{align*}
\mathcal{F}_\alpha \ : \  B_h[g] & \longrightarrow \R \\  B & \longmapsto \sum_{b\in B} b^{-\alpha} .
\end{align*}
is continuous. But if $z$ is such that $|z| <1$ and $B\in B_h[g]$, we have
$$   (f_B(z))^h = \sum_{(b_1,...,b_h) \in B^h} {z^{b_1+\dots+b_h}} 
         = \sum_{k \in \mathbb{N}} {R_{h,B}(k)z^k},$$
 where $R_{h,B}=|\{(b_1, \dots, b_h) \in B^h : \sum_i{b_i}=k\}|.$ By the hypothesis on $B$, there are only $g$ such distinct $h$-tuples up to permutation, and thus $ R_{h,B}(k) \leq h! g$. Finally
 $$  |(f_B(z))^h|  \leq gh!\sum_{k \in \mathbb{N}} {|z|^k}\leq \frac{h!g}{1-|z|},$$
which shows the continuity of $\mathcal{F}_\alpha$ and concludes the proof of the lemma.
\end{proof}
We can obviously extend the existence result of Corollary \ref{Bhg} to any $\alpha$, but we don't know if the supremum is finite for $\alpha=1/h$. Indeed, the best-known is a greedy algorithm \cite{limsupBHg}, which gives $\{a_n\}\in B_h[g]$ such that $a_n\ll n^{h+(h-1)/g}$.

\begin{question}
    Let $g\geq 2$ and $h \geq 2$, is
    $\sup\left\{ \sum_{b\in B} b^{-1/h} \ : \ B\in B_h[g]\right\}$ finite?
\end{question}

\printbibliography

\end{document}